\documentclass[11pt]{article}
\usepackage{ascmac}
\usepackage{fancybox}
\usepackage{amsmath, amsthm, amscd, amsfonts, amssymb, graphicx,enumerate,color}
\theoremstyle{definition}
\newtheorem{theorem}{Theorem}[section]

\newtheorem{corollary}[theorem]{Corollary}
\newtheorem{proposition}[theorem]{Proposition}
\newtheorem{remark}[theorem]{Remark}

\begin{document}
\title{More results on weighted means}
\author{Shigeru Furuichi$^1$\footnote{E-mail:furuichi.shigeru@nihon-u.ac.jp} and Mehdi Eghbali Amlashi$^2$\footnote{E-mail:amlashi@mshdiau.ac.ir}\\
$^1${\small Department of Information Science,}\\
{\small College of Humanities and Sciences, Nihon University,}\\
{\small 3-25-40, Sakurajyousui, Setagaya-ku, Tokyo, 156-8550, Japan}\\
$^2${\small Department of Mathematics, Mashhad Branch, Islamic Azad University,
Mashhad, Iran}}
\date{}
\maketitle
{\bf Abstract.} 
We give a refined Young inequality which generalizes the inequality by Zou--Jiang. We also show the upper bound for the logarithmic mean by the use of the weighted geometric mean and the weighted arithmetic mean.   Furthermore, we show some inequalities among the weighted means. Based on the obtained essential scalar inequalities, we give some operator inequalities. In particular, we give  a generalization of the result by Zou--Jiang, that is, we show the operator inequalities with the  operator relative entropy  with the weighted parameter. Finally, we give the further generalized inequalities by the Tsallis operator relative entropy. 
\vspace{3mm}

\vspace{3mm}

\vspace{3mm}


\section{Introduction}
	
		The logarithmic mean is defined by
\begin{equation} \label{lm_def}
L(a,b) := \frac{a-b}{\log a-\log b}= \int _0^1 a^tb^{1-t}dt,\quad a \neq b
\end{equation}
 for two positive numbers $a$ and $b$. (We usually define $L(a,b) = a$, if $a=b$.)
 It is known  the  inequality:
 \begin{equation}\label{intro_ineq01}
 L(a,b)  \leq \frac{2}{3}G(a,b)+ \frac{1}{3}A(a,b),
\end{equation}
which is called the classical P\'olya inequality in  \cite{Zou}, 
where $A(a,b):=\dfrac{1}{2}(a+b)$, $G(a,b):=\sqrt{ab}$.
		It is also known  the  inequality \cite{Lin1974}:
 \begin{equation}\label{intro_ineq02}
		 G(a,b) \le L(a,b) \le \left(\frac{a^{1/3}+b^{1/3}}{2}\right)^3.
\end{equation}
 We have the following relation \cite[Lemma 1.1]{FY2013}:
  \begin{equation}\label{intro_ineq03}
  		 L(a,b) \le \left(\frac{a^{1/3}+b^{1/3}}{2}\right)^3\le \frac{2}{3}G(a,b)+ \frac{1}{3}A(a,b),
  \end{equation}
which is a refinement of the P\'olya inequality.

		Recently, the weighted  logarithmic mean is introduced in \cite{PSMA2016} as
		\begin{equation}\label{wlm_def}
L_v(a,b) := \frac{1}{\log a-\log b}\left(\frac{1-v}{v}(a-a^{1-v}b^v)+\frac{v}{1-v}(a^{1-v}b^v-b)\right)
\end{equation}
 and studied in \cite{FN2020,MF2021}. In \cite[Corollary 2.2]{FN2020} and \cite[Theorem 2.2]{MF2021}, the following inequality is shown.
 \begin{equation}\label{sec01_ineq01}
 L_v(a,b) \le \frac{1}{2}G_v(a,b)+\frac{1}{2}A_v(a,b),
 \end{equation}
 where $A_v(a,b):=(1-v)a+vb$ is the weighted arithmetic mean and $G_v(a,b):=a^{1-v}b^v$ is the weighted geometric mean.
		However, the following inequality does not hold in general.
		 \begin{equation}\label{sec01_ineq02}
 L_v(a,b) \le \frac{2}{3}G_v(a,b)+\frac{1}{3}A_v(a,b),
 \end{equation}
		since we have counter-examples. See \cite{PSMA2016} for example.
		
We have the following relation for $a,b>0$ and $0\le v \le 1$,
		 \begin{equation}\label{order_four_means}
H_v(a,b) \le G_v(a,b)  \le L_v(a,b) \le A_v(a,b),
 \end{equation}
where $H_v(a,b):=\left\{(1-v)a^{-1}+vb^{-1}\right\}^{-1}$ is the weighted harmonic mean. The inequality $G_v(a,b)   \le A_v(a,b)$ is often called the Young inequality.
Many refinements and reverses for this inequality have been studied. See \cite[Chapter 2]{FM2020} for example. See also \cite{IA2021,K2017} for the recent advanced results.  In the paper \cite{FYM}, one of authors studied some inequalities on the weighted means, especially the weighted logarithmic mean. In this paper, we give the further results on the weighted mean and obtain some new inequalities for them.

\section{Main results}

We firstly give a new refinement of the Young inequality which is a generalization for the known result by using the weighted parameter $v \in [0,1]$.
\begin{theorem}\label{theorem_refine_Young01}
For $a,b>0$ and $0\le v \le 1$, we have
\begin{equation}\label{refine_Young_ineq01}
G_v(a,b)\le \left\{1+\frac{\mu^2}{2}\left(\log a-\log b\right)^2\right\}G_v(a,b) \le A_v(a,b),
\end{equation}
where $\mu:=\min\{1-v,v\}$.
\end{theorem}

\begin{proof}
The first inequality of \eqref{refine_Young_ineq01} is trivial.
To prove the second inequality of \eqref{refine_Young_ineq01}, we assume $0\le v \le \dfrac{1}{2}$ and set
$$
f_v(t):=(1-v)t+v-t^{1-v}-\frac{v^2}{2}\left(\log t\right)^2t^{1-v},\quad t>0.
$$
Then we have 
$$\frac{df_v(t)}{dt}=\frac{g_v(t)}{2t^v},\quad g_v(t):=2(1-v)(t^v-1)-2v^2\log t-v^2(1-v)\left(\log t\right)^2.$$
We also have 
$$
\frac{dg_v(t)}{dt}=\frac{2v}{t}h_v(t),\quad h_v(t):=(1-v)(t^v-\log t^v)-v.
$$
Since $\dfrac{dh_v(t)}{dt}=\dfrac{v(1-v)(t^v-1)}{t}$, we have $h_v(t)\ge h_v(1)=1-2v\ge 0$ which implies $\dfrac{dg_v(t)}{dt}\ge 0$
Thus we have $g_v(t)\ge g_v(1)=0$ for $t \ge 1$, and $g_v(t)\le g_v(1)=0$ for $0<t\le 1$, so that we have  $\dfrac{df_v(t)}{dt}\ge 0$ for $t \ge 1$, and $\dfrac{df_v(t)}{dt}\le  0$ for $0<t\le 1$. Therefore we have $f_v(t)\ge f_v(1)=0$. For the case $1/2 \le v \le 1$ can be proven similarly. Finally, putting $t:=a/b$ and then multiplying $b>0$ to both sides, we obtain the desired result.
\end{proof}

\begin{remark}
It is remarkable that Theorem \ref{theorem_refine_Young01} recovers the following inequality \cite{ZJ2015}:
\begin{equation}\label{refine_Young_ineq_remark01}
\left\{1+\frac{1}{8}\left(\log a-\log b\right)^2\right\}G(a,b)\le A(a,b)
\end{equation}
when $v=1/2$. In addition, the following reverse of the second inequality in \eqref{refine_Young_ineq01}:
$$
A_v(a,b)\le \left\{1+\frac{\lambda^2}{2}\left(\log a-\log b\right)^2\right\}G_v(a,b),\quad \lambda:=\max\{1-v,v\}
$$
does not hold in general, because of \eqref{refine_Young_ineq_remark01}.

The authors have been unaware the paper \cite{K2017} until the review. Although our generalization given in Theorem \ref{theorem_refine_Young01} for \eqref{refine_Young_ineq_remark01} is natural and simple, we have to point out that the inequality given in \cite[(2.1)]{K2017} is better than our one.
\end{remark}

Considering $r$-logarithmic function which is defined by $\ln_r x:=\dfrac{x^r-1}{r}$ for $x>0$ and $r\neq 0$, we find the following corollary of Theorem \ref{theorem_refine_Young01}. We note that $\lim\limits_{r\to 0}\ln_r x=\log x$.

\begin{corollary}\label{refine_Young_cor}
Let $0\le v \le 1$, $r\neq 0$ and $a,b>0$.
If we have the conditions (i)$r>0$ and $0<a\le b$, or (ii)$r<0$ and $a \ge b >0$, then we have the following inequalities.
\begin{equation}\label{refine_Young_cor_ineq01}
G_v(a,b)\le \left\{1+\frac{\mu^2}{2}\left(\ln_r\frac{a}{b}\right)^2\right\}G_v(a,b) \le A_v(a,b),
\end{equation}
where $\mu:=\min\{1-v,v\}$.
\end{corollary}

\begin{proof}
The first inequality of \eqref{refine_Young_cor_ineq01} is trivial.
We consider the function
$$
f_{v,r}(t):=(1-v)t+v-t^{1-v}-\frac{\mu^2}{2}(\ln_r t)^2t^{1-v},\quad (t>0,\quad r\neq 0,\quad 0\le v \le 1).
$$
Then we have
$$
\frac{df_{v,r}(t)}{dr}=\frac{\mu^2t^{1-v}(1-t^r)}{r^3}\left(1-t^r+t^r\log t^r \right).
$$
Putting $x:=1/t^r$
in the fundamental inequality $\log x \le x-1$ for $x>0$, we have $1-t^r+t^r\log t^r \ge 0$ for $t>0$ and $r\in\mathbb{R}$. Thus we have $\dfrac{df_{v,r}(t)}{dr} \ge 0$ for $r>0$ and $0<t\le 1$, which implies $f_{v,r}(t) \ge f_{v,0}(t) \ge 0$. The last inequality is thanks to  the second inequality of \eqref{refine_Young_ineq01} with the fact $\lim\limits_{r\to 0}\ln_r x=\log x$.
We also have $\dfrac{df_{v,r}(t)}{dr} \le 0$ for $r<0$ and $t\ge 1$, which implies $f_{v,r}(t) \ge f_{v,0}(t) \ge 0$, similarly. Therefore, if we have the conditions (i)$r>0$ and $0<t\le 1$, or (ii)$r<0$ and $t\ge 1$, then we have the following inequalities:
$$
(1-v)t+v-t^{1-v}-\frac{\mu^2}{2}(\ln_r t)^2t^{1-v}\ge 0.
$$
Putting $t:=a/b$ in the above and multiplying $b>0$ to both sides, we get the second inequality of \eqref{refine_Young_cor_ineq01}.
\end{proof}

It is well known that
\begin{equation}\label{Heinz_mean_ineq_orig}
G(a,b)\le Hz_v(a,b)\le A(a,b),
\end{equation}
where $Hz_v(a,b):=A(G_v(a,b),G_{1-v}(a,b))=\dfrac{a^{1-b}b^v+a^{v}b^{1-v}}{2}$ is the Heinz mean. Then we have the following inequalities from Theorem \ref{theorem_refine_Young01}.

\begin{corollary}\label{corollary_Heinz}
For $a,b>0$ and $0\le v \le 1$, we have
$$
G(a,b)\le Hz_v(a,b)\le \left\{1+\frac{\mu^2}{2}\left(\log a-\log b\right)^2\right\}Hz_v(a,b) \le A(a,b).
$$
\end{corollary}

\begin{proof}
Replacing $v$ by $1-v$ in \eqref{refine_Young_ineq01} and adding it to \eqref{refine_Young_ineq01} and then dividing 2, we get the result.
\end{proof}

The inequality
\begin{equation}\label{conc01}
L(a,b)\le \frac{1}{2}A_v(a,b)+\frac{1}{2}G_{v}(a,b),\quad a,b>0,\,\,\,\,0\le v \le 1
\end{equation}
does not hold in general, since we have counter-examples such as
\begin{equation}\label{conc02}
\frac{1}{2}A_{1/4}\left(1/2,1\right)+\frac{1}{2}G_{1/4}\left(1/2,1\right)-L\left(1/2,1\right)\simeq  -0.223091
\end{equation}
and
\begin{equation}\label{conc03}
\frac{1}{2}A_{3/4}\left(2,1\right)+\frac{1}{2}G_{3/4}\left(2,1\right)-L\left(2,1\right)\simeq  -0.446183.
\end{equation}
However, we have the following inequality for the weighted mean.
\begin{theorem}\label{theorem_heinz_ineq01}
For $a,b>0$ and $0\le v \le 1$, we have
\begin{equation}\label{heinz_ineq02}
L(a,b)\le \frac{1}{2}A_v(a,b)+\frac{1}{2}G_{1-v}(a,b).
\end{equation}
\end{theorem}
\begin{proof}
To prove \eqref{heinz_ineq02}, it is sufficient to prove $f_v(t)\geq 0$ for $t>0$ and $0\le v \le 1$, where
$$
f_v(t):=t^v+(1-v)t+v-\frac{2(t-1)}{\log t},\quad t>0.
$$
Then we have $\dfrac{f_v(t)}{dv}=1-t+t^v\log t$ and $\dfrac{d^2f_v(t)}{dv^2}=t^v\left(\log t\right)^2\ge 0$. Then $f_v(t)$ takes a minimum value at $v=v_{min}$, that is, 
$$\dfrac{f_{v_{min}}(t)}{dv}= 0 \Leftrightarrow t^{v_{min}}=\frac{t-1}{\log t}\Leftrightarrow v_{min}=\log_t\left(\dfrac{t-1}{\log t}\right).$$ 
Therefore we have 
$$
f_v(t)\ge f_{v_{min}}(t)=t-\frac{(t-1)}{\log t}\log\left(\frac{t-1}{\log t}\right)-\frac{(t-1)}{\log t}.
$$
We here prove $f_{v_{min}}(t)\ge 0$ for $t > 0$. Then we set for $t>0$ the function $g(t)$ by
$$
g(t):=\frac{t\log t}{t-1}-\log \left(\frac{t-1}{\log t}\right)-1.
$$
Then we have
$$
g'(t)=\frac{1}{t\log t}-\frac{\log t}{(t-1)^2}=\frac{(t-1)^2-t\left(\log t\right)^2}{t(t-1)^2\log t}\ge 0,\quad t\ge 1
$$
since we have $\dfrac{t-1}{\log t}\ge \sqrt{t}$ for $t>0$. 
We also have $g'(t)\le 0$ for $0<t \le 1$.
Thus we have $g(t)\ge g(1)=0$ for $t>0$, since $\lim\limits_{t\to 1} \dfrac{t\log t}{t-1}=1$ and $\lim\limits_{t\to 1} \dfrac{t-1}{\log t}=1$. Multiplying $\dfrac{t-1}{\log t} >0,\,\,(t>0)$ to $g(t)$,  we thus have  $f_v(t)\ge f_{v_{min}}(t)\ge 0$ for $t>0$.
\end{proof}

\begin{remark}
\begin{itemize}
\item[(i)] Replacing $v$ by $1-v$ in \eqref{heinz_ineq02} with same procedure in the Corollary \ref{corollary_Heinz}, we get the following inequality. 
\begin{equation}\label{heinz_ineq01}
L(a,b)\le \frac{1}{2}A(a,b)+\frac{1}{2}Hz_v(a,b).
\end{equation}
This inequality is also proven by the use of \eqref{sec01_ineq01} with $v=1/2$ and \eqref{Heinz_mean_ineq_orig} as
$$
L(a,b)\le \frac{1}{2}A(a,b)+\frac{1}{2}G(a,b)\le \frac{1}{2}A(a,b)+\frac{1}{2}Hz_v(a,b).
$$
\item[(ii)] 
From $g(t)\ge 0$ in the proof of Theorem \ref{theorem_heinz_ineq01} and \cite[Lemma 3]{F2012}, we have the bounds of $t^{\frac{t}{t-1}}$ in the following
$$
\frac{t-1}{\log t}\le \frac{t^{\frac{t}{t-1}}}{e}\le \frac{t^2+1}{t+1},\quad t>0.
$$
\end{itemize}
\end{remark}

We study the properties on the function which is a representing function of the weighted logarithmic mean.
\begin{equation}\label{Property_ineq01}
L_v(t):=\frac{1}{\log t}\left(\frac{1-v}{v}\left(t-t^{1-v}\right)+\frac{v}{1-v}\left(t^{1-v}-1\right)\right).
\end{equation}
We easily see $b L_v\left(a/b\right)=L_v(a,b)$, $L_{1/2}(t)=\dfrac{t-1}{\log t}$, $\lim\limits_{v\to 0}L_v(t)=t$ and $\lim\limits_{v\to 1}L_v(t)=1$. We also see $\lim\limits_{t\to 1}L_v(t)=1$ and $\lim\limits_{t\to 0}L_v(t)=0$. We have the following property.

\begin{proposition} \label{prop_monotone}
The function $L_v(t)$ given in \eqref{Property_ineq01} is increasing with respect to $v$ when $0<t\le 1$ and decreasing  with respect to $v$ when $t\ge 1$.
\end{proposition}

\begin{proof}
We calculate
$$
\frac{dL_v(t)}{dv}=\frac{f_v(t)}{v^2(1-v)^2t^{v-1}\log t},\,\,\,\,f_v(t):=-(1-v)^2t^v-v^2t^{v-1}+(2v^2-2v+1)+v(2v^2-3v+1)\log t.
$$
We also calculate
$$
\frac{df_v(t)}{dt}=\frac{v(1-v)}{t}g_v(t),\,\,\,\,g_v(t):=-(1-v)t^v+vt^{v-1}+1-2v.
$$
Since we have
$$
\frac{dg_v(t)}{dt}=-v(1-v)(t+1)t^{v-2} \le 0,
$$
we have $g_v(t)\ge g_v(1)=0$ for $0<t \le 1$ and $g_v(t)\le g_v(1)=0$ for $t\ge 1$. That is, we have $\dfrac{df_v(t)}{dt}\ge 0$ for  $0<t \le 1$ and $\dfrac{df_v(t)}{dt}\le  0$ for $t\ge 1$.  Thus we have $f_v(t)\le f_v(1)=0$.
Therefore we have $\dfrac{dL_v(t)}{dv} \ge 0$ for $0<t \le 1$ and   $\dfrac{dL_v(t)}{dv} \le 0$ for $t \ge 1$.
\end{proof}

Since $\dfrac{d}{dv}\left((1-v)t+v\right)=1-t$, $\dfrac{d}{dv}t^{1-v}=-t^{1-v}\log t$ and $\dfrac{d}{dv}\left(\dfrac{t}{(1-v)+v t}\right)=\dfrac{t(1-t)}{\left\{(1-v)+v t\right\}^2}$, we easily see that the representing functions for the weighted arithmetic mean, the weighted geometric mean and the weighted harmonic mean have similar properties. Proposition \ref{prop_monotone} will be applied to the proof of Proposition \ref{prop3.1}.

We further study the inequalities among means. In the sequel, we consider the bounds of the nested means for the weighted means. From the simple calculations and numerical computations, we see
$$
A(A_v(a,b),A_{1-v}(a,b))=A(a,b),\quad G(G_v(a,b),G_{1-v}(a,b))=G(a,b), 
$$
and
$$
H(H_v(a,b),H_{1-v}(a,b))=H(a,b),\quad L(L_v(a,b),L_{1-v}(a,b))\ne L(a,b).
$$
Because we have  
\begin{equation}\label{conjecture01}
L(L_{1/4}(10,1),L_{3/4}(10,1))- L(10,1)\simeq 0.0173327
\end{equation}
as an example.

As given in \eqref{order_four_means}, the distance between the geometric mean and the arithmetic mean is not so tight. Therefore the proof for the following relations are not so difficult.
\begin{proposition}\label{Relation_AG_prop01}
For $0\le v \le 1$ and $a,b>0$, we have
\begin{equation}\label{Relation_AG_ineq01}
G(a,b)\le G\left(A_v(a,b),A_{1-v}(a,b)\right) \le A(a,b)
\end{equation}
and
\begin{equation}\label{Relation_AG_ineq02}
G(a,b)\le A\left(G_v(a,b),G_{1-v}(a,b)\right)\le A(a,b).
\end{equation}
\end{proposition}
\begin{proof}
The inequalities \eqref{Relation_AG_ineq01} can be proven by
$$
A(a,b)=A(A_v(a,b),A_{1-v}(a,b))\ge G(A_v(a,b),A_{1-v}(a,b))\ge G(G_v(a,b),G_{1-v}(a,b))=G(a,b).
$$
The inequalities \eqref{Relation_AG_ineq01} are just same to the inequalities given in \eqref{Heinz_mean_ineq_orig}.
\end{proof}

From Proposition \ref{Relation_AG_prop01}, we have the following.
\begin{proposition}\label{Relation_HG_prop01}
For $0\le v \le 1$ and $a,b>0$, we have
\begin{equation}\label{Relation_HG_ineq01}
H(a,b)\le G\left(H_v(a,b),H_{1-v}(a,b)\right) \le G(a,b)
\end{equation}
and
\begin{equation}\label{Relation_HG_ineq02}
H(a,b)\le H\left(G_v(a,b),G_{1-v}(a,b)\right)\le G(a,b).
\end{equation}
\end{proposition}
\begin{proof}
Replacing $a$ and $b$ by $1/a$ and $1/b$ in Proposition \ref{Relation_AG_prop01} respectively, and taking inverse we have
\eqref{Relation_HG_ineq01} and \eqref{Relation_HG_ineq02}, since
$A(1/a,1/b)^{-1}=H(a,b)$, $G(1/a,1/b)^{-1}=G(a,b)$, 
$$G\left(A_v(1/a,1/b),A_{1-v}(1/a,1/b)\right)^{-1}=G\left(H_v(a,b),H_{1-v}(a,b)\right)$$
and $$A\left(G_v(1/a,1/b),G_{1-v}(1/a,1/b)\right)^{-1}=H\left(G_v(a,b),G_{1-v}(a,b)\right).$$
\end{proof}

We have the following relation on the arithmetic mean and the logarithmic mean. Their proofs  are not so easy, since the distance between the logarithmic mean and the arithmetic mean is relatively tight.

\begin{theorem}\label{Relation_Lv_prop01}
For $0\le v \le 1$ and $a,b>0$, we have
\begin{equation}\label{Relation_Lv_ineq01}
L(a,b)\le A\left(L_v(a,b),L_{1-v}(a,b)\right) \le A(a,b)
\end{equation}
and
\begin{equation}\label{Relation_Lv_ineq02}
L(a,b)\le L\left(A_v(a,b),A_{1-v}(a,b)\right)\le A(a,b).
\end{equation}
\end{theorem}

\begin{proof}
Since $L_v(a,b)\le A_v(a,b)$ and $L_{1-v}(a,b) \le A_{1-v}(a,b)$, we have
the second inequality of \eqref{Relation_Lv_ineq01}.
As for the first inequality of \eqref{Relation_Lv_ineq01} , it is sufficient to prove $f_v(t)\ge 0$ for $t>0$ and $0\le v \le 1$,
where $f_v(t):=\dfrac{1}{2}L_v(t)+\dfrac{1}{2}L_{1-v}(t)-L_{1/2}(t)$ and $L_v(t)$ is given in \eqref{Property_ineq01}. Since $\lim\limits_{v\to 0}L_v(t)=t$ and $\lim\limits_{v\to 1}L_v(t)=1$, for the special cases such as $v=0$ or $v=1$, $f_v(t)\ge 0$ is equivalent to $\dfrac{t+1}{2}\ge L(t)$ which is known. So we assume $0<v<1$.

By the simple calculations, we have
$$f_v(t)=\frac{g_v(t)}{2v(1-v)t^{v-1}\log t},\,\,\,\,\,\,g_v(t):=(1-2v)^2t^{v-1}(t-1)+(1-2v)(t^{2v-1}-1).$$
Then we have
$$
\frac{dg_v(t)}{dt}=v(1-2v)^2t^{v-2}h_v(t),\,\,\,\,\,\,h_v(t):=\frac{1-t^v}{v}-(1-t).
$$
Since $\dfrac{t^v-1}{v}\le t-1$ for $t>0$ and $0<v<1$, we have $h_v(t)\ge 0$ which implies $\dfrac{dg_v(t)}{dt}\geq 0$. Thus we have $g_v(t) \ge 0$ for $t \ge 1$ and $g_v(t)\le 0$ for $0<t\le 1$. Therefore we have $f_v(t)\ge 0$ for $t >0$ and $0\le v \le 1$.

The second inequality of \eqref{Relation_Lv_ineq02} can be easily proven by 
$$L(A_v(a,b),A_{1-v}(a,b))\le A(A_v(a,b),A_{1-v}(a,b))=A(a,b).$$
As for the first inequality of  \eqref{Relation_Lv_ineq02}, it is sufficient to prove the following inequality
\begin{equation}\label{Relation_Lv_ineq03}
\frac{t-1}{\log t}\le\frac{(1-2v)(t-1)}{\log\left\{(1-v)t+v\right\}-\log \left\{vt+(1-v)\right\}}
\end{equation}
for $t>0$ and $0\le v \le 1$.
Since the equality holds when $v=0,1/2,1$, we assume $0<v<1$ with $v\neq 1/2$.
\begin{itemize}
\item[(i)] For the case $t\ge 1$ and $0<v<1/2$, the inequality \eqref{Relation_Lv_ineq03} is equivalent to
$$
\log\frac{(1-v)t+v}{vt+(1-v)}\le (1-2v)\log t \Leftrightarrow (1-v)t+v\le vt^{2-2v}+(1-v)t^{1-2v}.
$$
So we set the function $k_v(t):=vt^{2-2v}+(1-v)t^{1-2v}-(1-v)t-v$.
Then we have $\dfrac{dk_v(t)}{dt}=2v(1-v)t^{1-2v}+(1-v)(1-2v)t^{-2v}-(1-v)$ and $\dfrac{d^2k_v(t)}{dt^2}=2v(1-v)(1-2v)t^{-2v-1}(t-1)\ge 0$.
Thus we have $\dfrac{dk_v(t)}{dt}\ge \dfrac{dk_v(1)}{dt}=0$ which implies
$k_v(t)\ge k_v(1)=0$. Therefore we have \eqref{Relation_Lv_ineq03} for $t\ge 1$  and $0<v<1/2$. Replacing $t:=1/s \ge 1$ in  \eqref{Relation_Lv_ineq03} and multiplying $s>0$ to both sides, we have
\begin{equation}\label{Relation_Lv_ineq04}
 \frac{s-1}{\log s}\le\frac{(1-2v)(1-s)}{\log\left\{\frac{(1-v)+vs}{s}\right\}-\log \left\{\frac{v+(1-v)s}{s}\right\}}
\end{equation}
for $0<s\le 1$ and $0<v<1/2$. 
Thus we have the inequality \eqref{Relation_Lv_ineq03} for $t>0$ and $0<v<\frac{1}{2}$. 
\item[(ii)] For the case $t\ge 1$ and $1/2<v<1$, the inequality \eqref{Relation_Lv_ineq03} is equivalent to
$$
\log\frac{(1-v)t+v}{vt+(1-v)}\ge  (1-2v)\log t \Leftrightarrow (1-v)t+v\ge  vt^{2-2v}+(1-v)t^{1-2v}.
$$
So we set the function $l_v(t):=(1-v)t+v-vt^{2-2v}-(1-v)t^{1-2v}$.
Then we have $\dfrac{dl_v(t)}{dt}=(1-v)-2v(1-v)t^{1-2v}-(1-v)(1-2v)t^{-2v}$ and $\dfrac{d^2l_v(t)}{dt^2}=-2v(1-v)(1-2v)t^{-2v-1}(t-1)\ge 0$.
Thus we have $\dfrac{dl_v(t)}{dt}\ge \dfrac{dl_v(1)}{dt}=0$ which implies
$l_v(t)\ge l_v(1)=0$. Therefore we have \eqref{Relation_Lv_ineq03} for $t\ge 1$  and $1/2<v<1$. By the similar way to the last part of (i), 
 we have the inequality \eqref{Relation_Lv_ineq03} for $t>0$ and $1/2<v<1$.  
\end{itemize}
From (i) and (ii), we have the inequality \eqref{Relation_Lv_ineq03} for $t>0$ and $0<v<1$.
\end{proof}

\begin{theorem}\label{Relation_GLv_prop01}
For $0\le v \le 1$ and $a,b>0$, we have
\begin{equation}\label{Relation_GLv_ineq01}
G(a,b)\le L\left(G_v(a,b),G_{1-v}(a,b)\right)\le L(a,b).
\end{equation}
and
\begin{equation}\label{Relation_GLv_ineq02}
G(a,b)\le G\left(L_v(a,b),L_{1-v}(a,b)\right) \le L(a,b)
\end{equation}
\end{theorem}

\begin{proof}
The first inequality of \eqref{Relation_GLv_ineq01} is easily proven by
$$L\left(G_v(a,b),G_{1-v}(a,b)\right) \ge G\left(G_v(a,b),G_{1-v}(a,b)\right)=G(a,b).$$

To prove the second inequality of \eqref{Relation_GLv_ineq01}, we
 set the function (since $L\left(G_v(a,b),G_{1-v}(a,b)\right)= L(a,b)$ for $v=1/2$)
 $$
 f_v(t):=(1-2v)(t-1)-t^{1-v}+t^v,\quad t\ge 1,\quad 0\le v<1/2.
 $$
 Since 
 $$
 \frac{df_v(t)}{dt}=(1-2v)-(1-v)t^{-v}+vt^{v-1},\quad  \frac{d^2f_v(t)}{dt^2}=v(1-v)t^{-v-1}(1-t^{2v-1})\geq 0,
 $$
 we have $ \frac{df_v(t)}{dt}\ge  \frac{df_v(1)}{dt}=0$ which implies
 $f_v(t)\ge f_v(1)=0$. By the similar way, we can prove
 $t^{1-v}-t^v-(1-2v)(t-1)\ge 0$ for $t \ge 1$ and $1/2<v\le 1$.
 Thus we have
 \begin{equation}\label{Relation_GLv_ineq03}
 \frac{t^{1-v}-t^v}{(1-2v)\log t}\le \frac{t-1}{\log t},\quad t\ge 1,\,\,\,0\le v \le 1.
 \end{equation}
 Putting $t:=1/s \ge 1$, we obtain
 $$
  \frac{s^{1-v}-s^v}{(1-2v)\log s}\le \frac{s-1}{\log s},\quad 0<s \le 1,\,\,\,0\le v \le 1.
 $$
 Thus we have
 \begin{equation}\label{Relation_GLv_ineq04}
  \frac{t^{1-v}-t^v}{(1-2v)\log t}\le \frac{t-1}{\log t},\quad t> 0,\,\,\,0\le v \le 1.
 \end{equation}
 Putting $t:=a/b$ in \eqref{Relation_GLv_ineq04} and multiplying $b>0$ to the both sides, we obtain the second inequality of \eqref{Relation_GLv_ineq01}.
 
 The first inequality of \eqref{Relation_GLv_ineq02} is easily proven by
$$G\left(L_v(a,b),L_{1-v}(a,b)\right) \ge G\left(G_v(a,b),G_{1-v}(a,b)\right)=G(a,b).$$
 
 Since $G(L_v(a,b),L_{1-v}(a,b))\le L(a,b)$ is equivalent to $L_v(t)L_{1-v}(t)\le L_{1/2}(t)^2$ because we put $t:=a/b>0$ and multiply $b^2>0$ to both sides. Since $(\log t)^2>0$ for all $t>0$, we set the function
 $$
 g_v(t):=(t-1)^2-\left\{\frac{1-v}{v} \left(t-t^{1-v}\right) +\frac{v}{1-v}\left(t^{1-v}-1\right)\right\}\left\{\frac{v}{1-v} \left(t-t^{v}\right) +\frac{1-v}{v}\left(t^{v}-1\right)\right\}.
 $$
 Since $g_v(t)=g_{1-v}(t)$ and $g_v(t)\ge 0,\,\,\,(t\ge 1)$ implies $g_v(s)\ge 0,\,\,\,(0< s\le 1)$ by putting $t:=1/s\ge 1$, we have only to prove $g_v(t)\ge 0$ for $t \ge 1$ and $0\le v \le 1/2$.
 Then we calculate
\begin{eqnarray*}
 &&\frac{dg_v(t)}{dt}=\frac{(1-2v)t^{-v-1}}{v^2(1-v)^2}\left\{(v-1)^3t+v^2(2-v)t^2+v^3t^{2v}-(1+v)(1-v)^2t^{2v+1}+2(1-2v)t^{v+1}\right\},\\
&& \frac{d^2g_v(t)}{dt^2}=\frac{(1-2v)t^{-v-2}}{v(1-v)}h_v(t),\quad
 h_v(t):=(1-v)^2t+v(2-v)t^2-v^2t^{2v}-(1-v)(1+v)t^{2v+1}.
\end{eqnarray*}
 We further calculate
 \begin{eqnarray*}
 && \frac{dh_v(t)}{dt}=(1-v)^2+2v(2-v)t-2v^3t^{2v-1}-(1-v)(1+v)(1+2v)t^{2v},\\
 && \frac{d^2h_v(t)}{dt^2}=2v\left(2-v\right)-2v(1-v)(1+v)(1+2v)t^{2v-1}+2v^3(1-2v)t^{2v},\\
 &&\frac{d^3h_v(t)}{dt^3}=2v(1-v)(1-2v)t^{2v-3}\left((v+1)(2v+1)t-2v^2\right)\ge 0,\quad \left(t\ge 1,\,\,\,\,0\le v \le 1/2\right).
 \end{eqnarray*}
 Thus we have $\dfrac{d^2h_v(t)}{dt^2} \ge \dfrac{d^2h_v(1)}{dt^2}=v(1-v)(1-2v)\geq 0$ which implies $\dfrac{dh_v(t)}{dt} \ge \dfrac{dh_v(1)}{dt}=0$.
 So we have $h_v(t)\ge h_v(1)=0$ which means $\dfrac{d^2g_v(t)}{dt^2}\ge 0$ which implies $\dfrac{dg_v(t)}{dt}\ge \dfrac{dg_v(1)}{dt}=0$. Therefore we have $g_v(t)\ge g_v(1)=0$.
\end{proof}

In the end of this section, we state some operator inequalities for the essential scalar inequalities which were obtained above. To this end, we give a notation for a self-adjoint operator $A$. If a self-adjoint operator $A$ satisfies $\langle Ax,x\rangle \ge 0$ for all vectors $x\neq 0$, then $A$ is called a positive operator, and we use the notation $A\ge 0$. If $\langle Ax,x\rangle > 0$ for all vectors $x\neq 0$, then $A$ is called strictly positive operator, and we use the notation $A>0$. It is known that the scalar order is equivalent to the operator partial order by Kubo-Ando theory \cite{KA}. Therefore it is often important to obtain a new scalar inequality. To express the logarithmic mean, we use the integral form such as
$L(t,1)=L_{1/2}(t)=\int _0^1 t^xdx$ and $L_v(t,1)=L_v(t)=\dfrac{v}{1-v}\int_0^{1-v} t^xdx +\dfrac{1-v}{v}\int_{1-v}^1 t^xdx$.
Since \eqref{Relation_GLv_ineq04} is rewritten as
$$
\frac{1}{1-2v}\int_v^{1-v}t^xdx \le \int _0^1 t^xdx,
$$
we have  for $A,B>0$
$$
\frac{1}{1-2v}\int_v^{1-v}A\sharp_xB dx \le \int_0^1 A\sharp_xB dx, \quad 0\le v \le 1
$$
by putting $t:=A^{-1/2}BA^{-1/2}$ and multiplying $A^{1/2}$ to both sides, where 
$$
A\sharp_xB := A^{1/2}\left(A^{-1/2}BA^{-1/2}\right)^xA^{1/2},\quad 0\le x \le 1
$$ 
is the weighted operator geometric mean for $A,B>0$. In the proof of Theorem \ref{Relation_GLv_prop01}, we proved $L_v(t)L_{1-v}(t)\le L_{1/2}(t)^2$ which is also rewritten as
$$
\left(\frac{v}{1-v}\int_0^{1-v}t^xdx +\frac{1-v}{v}\int_{1-v}^1t^xdx\right)
\left(\frac{1-v}{v}\int_0^{v}t^xdx +\frac{v}{1-v}\int_{v}^1t^xdx\right)\le \left(\int_0^1t^xdx\right)^2.
$$
Thus we similarly have  for $A,B>0$
$$
\left(A\ell_vB\right)A^{-1} \left(A\ell_{1-v}B\right)\le \left(A\ell B\right)A^{-1} \left(A\ell B\right),\quad 0\le v \le 1
$$
where the weighted operator logarithmic mean is defined by
$$
A\ell_vB:=\frac{v}{1-v}\int_0^{1-v}A\sharp_xBdx +\frac{1-v}{v}\int_{1-v}^1A\sharp_xBdx
$$
and the operator logarithmic mean is written by $A\ell B:=A\ell_{1/2}B= \int_0^1 A\sharp_xB dx$.
Since we proved $L_{1/2}(t)\le \dfrac{1}{2}L_v(t)+\dfrac{1}{2}L_{1-v}(t)$ for $t>0$ and $0\le v \le 1$,
we have  for $A,B>0$
$$
A\ell B \le \frac{1}{2}A\ell_vB+\frac{1}{2}A\ell_{1-v}B,\quad 0\le v \le 1.
$$

From \eqref{theorem_heinz_ineq01}, we also have for $A,B>0$
$$
A\ell B\le \frac{1}{2}A\nabla_vB+\frac{1}{2}A\sharp_{1-v}B,\quad 0\le v \le 1,
$$
where $A\nabla_v B:=(1-v)A+vB$ is the weighted operator arithmetic mean.
From \eqref{refine_Young_ineq01}, we have for $A,B>0$
\begin{equation}\label{generalization_ZJ}
0\le K^*\left(A\sharp_vB\right)K\le A\nabla_vB-A\sharp_vB,\quad 0\le v \le 1,
\end{equation}
where $K:=\dfrac{\mu}{\sqrt{2}}A^{-1}S(A|B)$, $\mu:=\{1-v,v\}$ and $S(A|B):=A^{1/2}\log \left(A^{-1/2}BA^{-1/2}\right)A^{1/2}$ is known as the operator relative entropy \cite{FK1989}. We see that the second inequality in \eqref{generalization_ZJ} gives a generalization of \cite[Theorem 4.1]{ZJ2015}.
Furthermore, the inequalities \eqref{refine_Young_cor_ineq01} is equivalent to the inequalities:
$$
0\le \frac{\mu^2}{2}(\ln_rt)t^v(\ln_rt)\le (1-v)+vt-t^v
$$
under the conditions (i)$r>0$ and $0<t \le 1$, or (ii)$r<0$ and $t \ge 1$.
Therefore we have the following proposition.
\begin{proposition}
Under the conditions (i)$r>0$ and $0<B\le A$, or (ii)$r<0$ and $0<A\le B$, we have the following operator inequalities:
$$
0\le K_r^*\left(A\sharp_vB\right)K_r\le A\nabla_vB-A\sharp_vB,\quad 0\le v \le 1,
$$ 
where $K_r:=\dfrac{\mu}{\sqrt{2}}A^{-1}S_r(A|B)$, $\mu:=\{1-v,v\}$ and $S_r(A|B):=A^{1/2}\ln_r \left(A^{-1/2}BA^{-1/2}\right)A^{1/2}$ is known as the Tsallis operator relative entropy \cite{YKF2005}.
\end{proposition}

The other obtained scalar inequalities give the corresponding operator inequalities. However, we omit them.

\section{Concluding remarks}
Related to \eqref{conc01}, we have the following result which does not contradicts with \eqref{conc02} and \eqref{conc03}.
\begin{proposition}\label{prop3.1}
For $a,b>0$ and $0\le v \le 1$, we have the following inequalities.
\begin{itemize}
\item[(i)] For $0\le v \le 1/2$ and $a \ge b$, we have
\begin{equation}\label{conc04}
L(a,b)\le\frac{1}{2}A_v(a,b)+\frac{1}{2}G_v(a,b).
\end{equation}
\item[(ii)] $ 1/2\le v \le 1$ and $a \le b$, we also have
\begin{equation}\label{conc05}
L(a,b)\le\frac{1}{2}A_v(a,b)+\frac{1}{2}G_v(a,b).
\end{equation}
\end{itemize}
\end{proposition}
\begin{proof}
\begin{itemize}
\item[(i)] For the case $a/b=:t \ge 1$, from Proposition \ref{prop_monotone}, for $0\le v \le 1/2$ we have $L_{1/2}(t,1)\le L_v(t,1)$ which implies $L(a,b)=L_{1/2}(a,b)\le L_v(a,b)$. Thus we have $L(a,b)\le L_v(a,b)\le \dfrac{1}{2}A_v(a,b)+\dfrac{1}{2}G_v(a,b)$ for $0\le v \le 1/2$ from  \eqref{sec01_ineq01}.
\item[(ii)] For the case $a/b=:t \le 1$, from Proposition \ref{prop_monotone}, we have similarly $L(a,b)=L_{1/2}(a,b)\le L_v(a,b)$ for $1/2\le v \le 1$. Thus we have $L(a,b)\le L_v(a,b)\le \dfrac{1}{2}A_v(a,b)+\dfrac{1}{2}G_v(a,b)$ for $1/2\le v \le 1$ from  \eqref{sec01_ineq01}.
\end{itemize}
\end{proof}

It is quite natural to consider the maximum (optimal) value $p$ such that
\begin{equation}\label{sec01_ineq03}
 L_v(a,b) \le (1-p) A_v(a,b)+p G_v(a,b),
\end{equation}

By the numerical computation shows that
$$(1-p) A_v(a,b)+p G_v(a,b)- L_v(a,b) \simeq -1.39948\times 10^{-8}$$
when $a:=10^{-10}$, $b:=1$, $v:=1-10^{-10}$ and $p=\dfrac{13}{25}$. This means the inequality \eqref{sec01_ineq03} does not hold when $p=\dfrac{13}{25}>\dfrac{1}{2}.$

We close this paper with the following conjecture.
From \eqref{conjecture01} with several numerical computations indicate
that the inequality seems to be true
\begin{equation}\label{conjecture02}
L(a,b)\le L(L_v(a,b),L_{1-v}(a,b)),\quad a,b>0,\quad 0\le v \le 1. 
\end{equation}
However we have not proven this inequality due to its complicated computations, and we also have not found any counter-examples.
If the conjectured inequality \eqref{conjecture02} will be shown, then it will give a tight inequality for the first inequalities in \eqref{Relation_Lv_ineq01} and \eqref{Relation_Lv_ineq02}.

\section*{Acknowledgement}
The author (S.F.) was partially supported by JSPS KAKENHI Grant Number 21K03341.

\section*{Appendix: Proof of \eqref{sec01_ineq01} by elementary calculations}

As we noted, \eqref{sec01_ineq01} has been proven already by the use of the Hermite-Hadamard inequality for convex function. From the point of self-sufficiency,  we give a direct and an elementary proof for \eqref{sec01_ineq01}.

{\it Proof of \eqref{sec01_ineq01}:}
To prove \eqref{sec01_ineq01}, we firstly prove $f_v(t) \ge 0$ for $t \ge 1$ and $0\le v \le 1$, where
$$
f_v(t):=v(1-v)t^{1-v}\log t+v(1-v)\left\{(1-v)t+v\right\}\log t-2\left\{(1-v)^2(t-t^{1-v})+v^2(t^{1-v}-1)\right\}.
$$
Then we have $f_v'(t)=\dfrac{(1-v)}{t^v}g_v(t)$, where
$$
g_v(t):=(3v-2-v^2)t^v+v^2t^{v-1}+(2-3v)+v(1-v)(1+t^v)\log t.
$$
Then we also have $g_v'(t)=\dfrac{v(1-v)}{t}h_v(t)$, where
$$
h_v(t):=1-vt^{v-1}-(1-v)t^v+vt^v\log t.
$$
Since we have $h_v'(t)=vt^{v-2}\left\{(1-v)+vt+vt\log t\right\}\geq 0$ for 
$t \ge 1$ and $0\le v \le 1$, $h_v(t)\ge h_v(1)=0$. Thus we have $g_v'(t)\ge 0$ which implies $g_v(t)\ge g_v(1)=0$. Therefore we have $f_v'(t)\ge 0$ so that we have $f_v(t)\ge f_v(1)=0$.

We secondly prove $k_v(t)\ge 0$ for $0< t \le 1$ and  $0\le v \le 1$, where
$$
k_v(t):=2\left\{(1-v)^2(t-t^{1-v})+v^2(t^{1-v}-1)\right\}-v(1-v)t^{1-v}\log t-v(1-v)\left\{(1-v)t-v\right\}\log t.
$$
Then we have $k_v'(t)=\dfrac{(1-v)}{t^v}l_v(t)$, where
$$
l_v(t):=(2-3v+v^2)t^v-v^2t^{v-1}+(3v-2)-v(1-v)(1+t^v)\log t.
$$
Then we also have $l_v'(t)=\dfrac{v(1-v)}{t}\left(m_v(t)-vt^v\log t\right)$,
where
$$
m_v(t):=(1-v)t^v+vt^{v-1}-1.
$$
Since $m_v'(t)=v(1-v)t^{v-2}(t-1)\le 0$ for $0< t \le 1$ and $0\le v \le 1$, we have $m_v(t)\ge m_v(1)=0$. Thus we have $l_v'(t)\ge 0$  for $0< t \le 1$ and $0\le v \le 1$ so that we have $l_v(t)\le l_v(1)=0$ which implies $k_v'(t) \le 0$. Therefore   for $0< t \le 1$ and $0\le v \le 1$, we have $k_v(t)\ge k_v(1)=0$.
\hfill \qed

\end{document}